\documentclass[a4paper,11pt]{article}
\usepackage[cp1251]{inputenc}
\usepackage[english]{babel}
\usepackage{amsmath,amssymb,amsthm}
\usepackage{hyperref}

\newcommand{\R}{\mathbb{R}}

\newcommand{\N}{\mathbb{N}}

\newcommand{\cM}{\mathcal{M}}

\newcommand{\wh}{\widehat}
\newcommand{\crt}{\operatorname{crt}}
\newcommand{\id}{\operatorname{id}}
\newcommand{\mob}{\operatorname{Mob}}

\newcommand{\set}[2]{\{#1:\,\text{#2}\}}

\renewcommand{\phi}{\varphi}
\renewcommand{\epsilon}{\varepsilon}

\theoremstyle{plain}
\newtheorem{theorem}{Theorem}
\newtheorem{corollary}{Corollary}
\newtheorem{lemma}{Lemma}
\newtheorem{proposition}{Proposition}

\theoremstyle{definition}

\theoremstyle{remark}

 \newtheoremstyle{break}
   {9pt}
   {9pt}
   {\itshape}
   {}
   {\bfseries}
   {.}
   {\newline}
   {}

 \theoremstyle{break}

\title{Ptolemy spaces with strong inversions}
\author{Alexander Smirnov\footnote{Supported by RFFI Grant
11-01-00302-a}}

\begin{document}

\date{}
\maketitle

\begin{abstract}
We prove that a compact Ptolemy space with many strong inversions 
that contains a Ptolemy circle is M\"obius equivalent to an extended
Euclidean space. 
\end{abstract}

\section{Introduction}
This paper was motivated by the works~\cite{BS1} and \cite{BS2} of 
S.~Buyalo and V.~Schroeder giving a M\"obius characterization of the boundary
at infinity of the rank one symmetric spaces of noncompact type.
Their characterization uses the notion of a {\em space inversion}, 
w.r.t. distinct
$\omega$, $\omega'\in X$
and a metric sphere 
$S\subset X$
between
$\omega$, $\omega'$,
which is  a M\"obius automorphism
$\phi=\phi_{\omega,\omega',S}:X\to X$
such that
\begin{itemize}
 \item[(1)] $\phi$
is an involution,
$\phi^2=\id$,
without fixed points;
 \item[(2)] $\phi(\omega)=\omega'$ (and thus
$\phi(\omega')=\omega$);
 \item[(3)] $\phi$
preserves
$S$,
$\phi(S)=S$;
 \item[(4)] $\phi(\sigma)=\sigma$
for any Ptolemy circle
$\sigma\subset X$
through
$\omega$, $\omega'$.
\end{itemize}

Recall that however a classical inversion of an Euclidean space
$\R^n$
with respect to a sphere
$S\subset\R^n$
fixes
$S$
pointwise. In this paper we impose on an s-inversion a stronger condition
that $\phi$
preserves
$S$
pointwise, 
$\phi(x)=x$
for every 
$x\in S$, and call it {\em strong s-inversion}.
We study Ptolemy spaces with two following properties. 

\begin{itemize}
 \item[(E)] Existence: there is at least one Ptolemy circle in $X$.
 \item[(sI)] strong Inversion: for any distinct $\omega, \omega'\in X$ and 
any metric sphere $S\subset X$ 
between $\omega, \omega'$ there is a strong space inversion 
$\phi_{\omega,\omega',S}\colon X\to X$ 
w.r.t. $\omega, \omega'$ and S.
\end{itemize}
 
Our main result is the proof of the following theorem. 

\begin{theorem}\label{main:thm} Let
$X$
be a compact Ptolemy space with properties 
(E) and (sI). Then
$X$
is M\"obius equivalent to the extended Euclidean space
$\wh\R^n=\R^n\cup\{\infty\}$
for some
$n\ge 1$.
\end{theorem}

Another M\"obius characterization of
$\wh\R^n$
is obtained in \cite{FS}: {\em a compact Ptolemy space
$X$
is M\"obius equivalent to
$\wh\R^n$
if and only if through any three points in
$X$
there is a Ptolemy circle.}

Despite the differences in the definition of s-inversions and strong s-inversions, 
some properties of studied spaces hold in both cases.
Thus the definitions of homotheties and shifts, as well as  
Lemmas \ref{lem:unique_line}, 
\ref{lem:busparallel_sublinear}, 
\ref{lem:shift_busemann_parallel} 
are originally presented in \cite{BS1}. 
The significant differences between the classes of such spaces 
arise when we consider a symmetry w.r.t. a horosphere. 
In general, if we assume only the existence of s-inversions there is no reason 
that such a symmetry exist. 

{\em Acknowledgements}. 
I would like to thank Sergei Buyalo for many useful
advices and attention to this paper.

\section{Basic definitions}

\subsection{M\"obius structures}

In this section we will follow definitions from \cite{BS1}.  
Namely, fix a set $X$ and consider {\em extended} metrics on
$X$
for which existence of an {\em infinitely remote} point
$\omega\in X$
is allowed, that is,
$d(x,\omega)=\infty$
for all
$x\in X$, $x\neq\omega$. 
We always assume that such a point is unique if exists, and that
$d(\omega,\omega)=0$.

A quadruple
$Q=(x,y,z,u)$
of points in a set
$X$
is said to be {\em admissible} if no entry occurs three or
four times in 
$Q$.
Two metrics 
$d$, $d'$
on 
$X$ 
are {\em M\"obius equivalent} if for any admissible quadruple
$Q=(x,y,z,u)\subset X$
the respective {\em cross-ratio triples} coincide,
$\crt_d(Q)=\crt_{d'}(Q)$,
where
$$\crt_d(Q)=(d(x,y)d(z,u):d(x,z)d(y,u):d(x,u)d(y,z))\in\R P^2.$$
If 
$\infty$ 
occurs once in 
$Q$, 
say 
$u=\infty$,
then
$\crt_d(x,y,z,\infty)=(d(x,y):d(x,z):d(y,z))$.
If 
$\infty$ 
occurs twice, say 
$z=u=\infty$, 
then
$\crt_d(x,y,\infty,\infty)=(0:1:1)$.

A {\em M\"obius structure} on a set
$X$
is a class 
$\cM=\cM(X)$
of metrics on
$X$
which are pairwise M\"obius equivalent.

The topology considered on 
$(X,d)$ 
is the topology with the basis consisting of all open distance balls 
$B_r(x)$
around points in 
$x\in X_\omega$ 
and the complements $X\setminus D$ 
of all closed distance balls 
$D=\overline{B}_r(x)$. 
M\"obius equivalent metrics define
the same topology on
$X$.
When a M\"obius structure 
$\cM$
on
$X$
is fixed, we say that
$(X,\cM)$
or simply
$X$
is a {\em M\"obius space.}

A map
$f:X\to X'$
between two M\"obius spaces
is called {\em M\"obius}, if 
$f$ 
is injective and for all admissible quadruples
$Q\subset X$
$$\crt(f(Q))=\crt(Q),$$
where the cross-ratio triples are taken with respect to
some (and hence any) metric of the M\"obius structures
of
$X$, $X'$.
M\"obius maps are continuous. If a M\"obius map
$f:X\to X'$
is bijective, then 
$f^{-1}$
is M\"obius,
$f$
is homeomorphism, and the M\"obius
spaces
$X$, $X'$
are said to be {\em M\"obius equivalent}. 

We note that if two M\"obius equivalent 
metrics have the same infinitely remote point, 
then they are homothetic, see e.g. \cite{BS1, FS}.

A classical example of a M\"obius space is the extended
$\wh\R^n=\R^n\cup\infty=S^n$, $n\ge 1$,
where the M\"obius structure is generated by some extended
Euclidean metric on
$\wh\R^n$,
and
$\R^n\cup\infty$
is identified with the unit sphere
$S^n\subset\R^{n+1}$
via the stereographic projection.

\subsection{Ptolemy spaces}

A M\"obius space
$X$
is called a {\em Ptolemy space}, if it satisfies the
Ptolemy property, that is, for all admissible quadruples  
$Q\subset X$
the entries of the respective cross-ratio triple
$\crt(Q)\in\R P^2$
satisfy the triangle inequality.

The Ptolemy property is equivalent to that the M\"obius structure
$\cM$ 
of
$X$ 
is invariant under metric inversions, or in other words,
$\cM$
is Ptolemy if and only if for all
$z\in X$ 
there exists a metric
$d_z\in\cM$ 
with infinitely remote point
$z$. 

Recall that the metric inversion (or m-inversion for brevity) 
of a metric
$d\in\cM(X)$
w.r.t.
$z\in X\setminus\omega$
($\omega$ 
is infinitely remote for 
$d$)
of radius
$r>0$
is a function
$d_z(x,y)=\frac{r^2d(x,y)}{d(z,x)d(z,y)}$
for all
$x$, $y\in X$
distinct from
$z$, $d_z(x,z)=\infty$
for all
$x\in X\setminus\{z\}$
and
$d_z(z,z)=0$.

The classical example of a Ptolemy space is
$\wh\R^n$
with a standard M\"obius structure. 
 
One of the first interesting facts about Ptolemy spaces is the Schoenberg theorem.  

\begin{theorem}[Schoenberg, \cite{Sch}]\label{thm:schoenberg}
A real normed vector space, which is a Ptolemy space, is an inner product space.  
\end{theorem}

A Ptolemy circle in a Ptolemy space 
$X$ 
is a subset 
$\sigma\subset X$ 
homeomorphic to 
$S^1$  
such that for every quadruple
$(x,y,z,u)\in\sigma$
of distinct points the equality 
\begin{equation*}\label{eq:PT_eq}
d(x,z)d(y,u)=d(x,y)d(z,u)+d(x,u)d(y,z)
\end{equation*}
holds for some and hence for any metric 
$d$
of the M\"obius structure, where it is supposed that the pair
$(x,z)$
separates the pair
$(y,u)$,
i.e.
$y$
and
$u$
are in different components of
$\sigma\setminus\{x,z\}$. 

Given
$\omega\in X$,
we use the notation
$X_\omega=X\setminus\omega$
and always assume that a metric of the M\"obius structure on
$X_\omega$
is fixed. 
Note that every Ptolemy circle
$\sigma\subset X$
that passes through
$\omega$
is isometric to a geodesic line in $X_\omega$. Such a line
$\ell=\sigma_\omega$ 
is called a {\em Ptolemy} line.

\subsection{Space inversions}

Given distinct
$\omega$, $\omega'\in X$,
we say that a subset
$S\subset X$
is a {\em metric sphere between}
$\omega$, $\omega'$,
if
$$S=\set{x\in X}{$d(x,\omega)=r$}=S_r^d(\omega)$$
for some metric
$d\in\cM$
with infinitely remote point 
$\omega'$
and some
$r>0$.
Any two such metrics
$d$, $d'\in\cM$
are proportional to each other,
$d'=\lambda d$
for some 
$\lambda>0$. 
Then
$S_r^d(\omega)=S_{\lambda r}^{d'}(\omega)$.
Moreover, this notion is symmetric w.r.t. 
$\omega$, $\omega'$,
because any metric 
$d'\in\cM$
with infinitely remote point 
$\omega$
is proportional to the m-inversion of
$d$
w.r.t.
$\omega$,
and we can assume that
$d'$
is the m-inversion itself. Then
$S=\set{x\in X}{$d'(x,\omega')=1/r$}$.

We define a {\em strong space inversion}, or s-inversion for brevity,
w.r.t. distinct
$\omega$, $\omega'\in X$
and a metric sphere 
$S\subset X$
between
$\omega$, $\omega'$
as a M\"obius automorphism
$\phi=\phi_{\omega,\omega',S}:X\to X$
such that
\begin{itemize}
 \item[(1)] $\phi$
is an involution,
$\phi^2=\id$;
 \item[(2)] $\phi(\omega)=\omega'$ (and thus
$\phi(\omega')=\omega$);
 \item[(3)] $\phi$
preserves
$S$
pointwise, 
$\phi(x)=x$
for every 
$x\in S$;
 \item[(4)] $\phi(\sigma)=\sigma$
for any Ptolemy circle
$\sigma\subset X$
through
$\omega$, $\omega'$.
\end{itemize}

Let $\omega\in X$.
Fix 
$o\in X_\omega$ 
and consider a metric sphere $S=S_r(o)$ between $o$ and $\omega$. 
Let $\phi$ be an s-inversion w.r.t. $o,\omega$ and $S$.
Now we prove two technical lemmas. 

\begin{lemma}\label{lem:inv_1}
Let $x\in X_\omega$.   
Then 
$|ox|\cdot |o\phi(x)|=r^2$.
\end{lemma}
\begin{proof}
Let $y\in S$. 
Then 
\begin{multline*}
\crt(x,y,o,\omega)=(|xy|:|xo|:|yo|)=\crt(\phi(x),\phi(y),\phi(o),\phi(\omega))\\
=\crt(\phi(x),y,\omega,o)=(|\phi(x)y|:|yo|:|\phi(x)o|).
\end{multline*}
It follows that $|\phi(x)o|/|yo|=|yo|/|xo|$ and $|ox|\cdot |o\phi(x)|=r^2$. 
\end{proof}

\begin{lemma}\label{lem:inv_2}
Let $x,y\in X_\omega$. 
Then
$
|\phi(x)\phi(y)|=r^2\cdot\frac{|xy|}{|ox|\cdot|oy|}
$.  
\end{lemma}
\begin{proof}
Note that 
\begin{multline*}
\crt(x,y,o,\omega)=(|xy|:|xo|:|yo|)\\
=\crt(\phi(x),\phi(y),\omega,o)=
(|\phi(x)\phi(y)|:|\phi(y)o|:|\phi(x)o|).
\end{multline*}
It follows that $|\phi(x)\phi(y)|/|\phi(x)o|=|xy|/|yo|$. 
From Lemma~\ref{lem:inv_1} we have that $|\phi(x)o|=r^2/|xo|$. 
Then 
$
|\phi(x)\phi(y)|=|\phi(x)o|\cdot\frac{|xy|}{|yo|}=r^2\frac{|xy|}{|ox|\cdot|oy|}
$.
\end{proof}

We say that a M\"obius space $X$ has the property
(E)  if there is a Ptolemy circle in  
$X$.
And we also say that 
a M\"obius space $X$ has the property
(sI) if for any distinct
$\omega$, $\omega'\in X$
and a metric sphere
$S\subset X$
between
$\omega$, $\omega'$
there is an s-inversion
$\phi_{\omega,\omega',S}:X\to X$
w.r.t.
$\omega$, $\omega'$
and
$S$.

From now on, we assume that 
$X$ is a compact Ptolemy space with properties 
(E) and (sI). 

\section{Homotheties and shifts}

\subsection{Homotheties}

Fix $\omega\in X$. 
Let $o\in X_\omega$, $\lambda>0$.  
Consider $r_1,r_2>0$ such that $\lambda=r^2_2/r^2_1$.
Let  $S_1=S_{r_1}(o), S_2=S_{r_2}(o)\subset X_\omega$ be metric spheres between $o,\omega$. 
Denote by $\phi_1,\phi_2$  s-inversions w.r.t. $o, \omega, S_1$ and $o, \omega, S_2$ respectively.

We define a {\em homothety 
with the center $o$ and the coefficient $\lambda$}
as a M\"obius automorphism $h\colon X\to X$ such that 
$h=\phi_2\circ \phi_1$.

Note that the next properties follow from the definition of an s-inversion and 
from Lemma~\ref{lem:inv_2}.  
 
\begin{itemize}
 \item[(1)] $h(o)=o, h(\omega)=\omega$. 
 \item[(2)] $h(\sigma)=\sigma$
for any Ptolemy circle
$\sigma\subset X$
through
$o$, $\omega$.
 \item[(3)] $|h(x)h(y)|=\lambda |xy|$ for all $x, y \in X_{\omega}$. 
 \item[(4)] For each $o\in X_\omega$ and each $\lambda>0$  
there exists a homothety with the center $o$ and the coefficient $\lambda$.
 
\end{itemize}

We denote a homothety with the center $o$ and the coefficient $\lambda$ by $h_{\lambda,o}$. 

\begin{proposition}\label{pro:homothety_transit}
Let  $\omega$, $\omega'\in X$, 
and let
$\sigma$ be a Ptolemy circle
through $\omega$, $\omega'$ and let $\Gamma\subset \sigma$ be
a connected component of $\sigma\setminus\{\omega,\omega'\}$. 
Consider $x,x'\in \Gamma$.
Then there exists a homothety $h$ with center  $\omega'$ such that $h(x)=x'$. 
\end{proposition}

\begin{proof}
Consider a metric space $X_\omega$.  
Since $\omega\in \sigma$, 
$\Gamma$ is a geodesic ray starting at $\omega'$. 
Define
$\lambda$
by
$|\omega'x'|=\lambda|\omega'x|$. 
Then
$h(x)=x'$
for
$h=h_{\lambda,\omega'}$. 
\end{proof}

\begin{corollary}\label{cor:ptsirctwopoints}
Any two distinct Ptolemy circles in a Ptolemy space with
properties (E) and (sI) have at most two points in common.
\end{corollary}
\begin{proof}
Let $\sigma,\sigma'\subset X$ 
be intersecting Ptolemy circles with 
$\omega\in \sigma\cap \sigma'$. 
Consider a metric space $X_\omega$. 
By contradiction suppose that there exist
$x,x'\in (\sigma\cap \sigma')\setminus\{\omega\}$. 
Let $\Gamma$ be a connected component of $\sigma\setminus\{x,\omega\}$ such that $x'\in \Gamma$. 
Also let $\Gamma'$ be a connected component of $\sigma'\setminus\{x,\omega\}$ such that $x'\in \Gamma'$. 
Note that if $x''\in \Gamma$ and $\lambda=|xx''|/|xx'|$ then for a homothety $h=h_{\lambda,x}$  
we have $h(x')=x''$. Then $x''\in \Gamma'$ and $\Gamma\subset \Gamma'$. 
Similarly, $\Gamma'\subset \Gamma$ and thus $\Gamma=\Gamma'$. 
In the same way, if  $\Gamma_1$ is a connected component of $\sigma\setminus\{x',\omega\}$ 
such that $x\in \Gamma_1$  and $\Gamma'_1$ is a connected component of 
$\sigma'\setminus\{x',\omega\}$ such that $x\in \Gamma'_1$, 
we can prove that $\Gamma_1=\Gamma'_1$. 
It follows that $\sigma=\sigma'$. 
\end{proof}

\subsection{Shifts}

Note that $X$ is Hausdorff and compact. 
If we fix a nonprincipal ultrafilter $\theta$ 
on the set of natural numbers $\N$ then for each sequence $x_n\in X$ there exists a unique 
$x\in X$ such $x=\lim_\theta x_n$. 
Moreover $|\lim_\theta(x_n)\lim_\theta(y_n)|=\lim_\theta|x_ny_n|$ for all sequences $x_n, y_n\in X$. 

In this section we need the following well known fact, see e.g.
\cite{BS1}, Lemma~6.7. 

\begin{lemma}\label{lem:nondegenerate_morphism} Assume that
for a nondegenerate triple
$T=(x,y,z)\subset X$
and for a sequence 
$\phi_i\in\mob X$
the sequence
$T_i=\phi_i(T)$
$\theta$-converges to a nondegenerate triple
$T'=(x',y',z')\subset X$.
Then there exists
$\phi=\lim_\theta\phi_i\in\mob X$
with
$\phi(T)=T'$.
\end{lemma}

Fix $\omega\in X$ and let $x, x'\in X_{\omega}$. 
Let $\lambda_n>0$, $n\in \N$, be a sequence goes to zero. 
Consider a homothety $h_n$ with center $x$ and coefficient $\lambda^{-1}_n$ and 
a homothety $h'_n$ with center $x'$ and coefficient $\lambda_n$. 
Denote their composition $h'_n\circ h_n$ by $\eta_n$. 
Note that $\eta_n$ is an isometry for each $n\in \N$. 
Then by Lemma~\ref{lem:nondegenerate_morphism} $\eta=\lim_\theta\eta_n$   
is a M\"obius automorphism with
$\eta(x)=x'$ and $\eta(\omega)=\omega$.
Moreover 
$\eta:X_\omega\to X_\omega$
is an isometry. We call the isometry $\eta_{xx'}$ 
constructed above a {\em shift} from $x$ to $x'$. 
For each $x,x'\in X_\omega$ 
there exists a shift from $x$ to $x'$.

\section{Foliations by parallel lines} 

Each Ptolemy line 
$\ell\subset X_\omega$ 
is isometric to $\R$ so 
for every
$x_0\in\ell$
the Busemann functions 
$b^{\pm}_{\ell,x_0}\colon X_{\omega} \to \R$ are well defined 
by the formula 
$$
b^{\pm}_{\ell,x_0}(x)=\lim\limits_{t\to \pm\infty}|xc(t)|-|x_0c(t)|, 
$$ 
where $c(t)\colon \R\to \ell$ is a unit speed parameterization.
 
We say that Ptolemy lines
$\ell$, $\ell'\subset X_\omega$
are {\em Busemann parallel} if 
$\ell$, $\ell'$
share Busemann functions, that is, any Busemann function
associated with
$\ell$
is also a Busemann function associated with
$\ell'$
and vice versa.

The following lemmas are proved in \cite{BS1}, and the proofs
go without changes in our case.

\begin{lemma}[\cite{BS1}, Lemma~4.11]\label{lem:unique_line} Let
$\ell$, $\ell'\subset X_\omega$
be Ptolemy lines with a common point,
$o\in l\cap \ell'$, $b:X_\omega\to\R$ be 
a Busemann function of
$\ell$
with
$b(o)=0$.
Assume 
$b\circ c(t)=-t=b\circ c'(t)$
for all 
$t\ge 0$
and for appropriate unit speed parameterizations
$c$, $c':\R\to X_\omega$
of
$\ell$, $\ell'$
respectively
with
$c(0)=o=c'(0)$.
Then
$l=l'$.
In particular, Busemann parallel Ptolemy lines coincide if
they have a common point.
\end{lemma}

\begin{lemma}[\cite{BS1}, Lemma~4.12]\label{lem:busparallel_sublinear}
Let
$c$, $c':\R\to X_\omega$
be unit speed parameterizations of Ptolemy lines
$\ell$, $\ell'\subset X_\omega$
respectively. 
If
$|c(t_i)c'(t_i)|/|t_i|\to 0$
for some sequences
$t_i\to\pm\infty$,
then the lines
$\ell$, $\ell'$
are Busemann parallel.

Vice versa, if 
$\ell$, $\ell'\subset X_\omega$
are Busemann parallel lines
then
$$\lim\limits_{t\to \infty}|c(t)c'(t)|/t=0$$
for appropriately chosen their unit speed parameterizations
$c$, $c':\R\to X_\omega$.
\end{lemma}

\begin{lemma}[\cite{BS1}, Lemma~4.13]\label{lem:shift_busemann_parallel}
A shift 
$\eta_{xx'}$
moves any Ptolemy line
$\ell$
through
$x$
to a Busemann parallel Ptolemy line 
$\eta_{xx'}(l)$
through
$x'$.
\end{lemma}

From Lemma~\ref{lem:unique_line} and 
Lemma~\ref{lem:shift_busemann_parallel} we immediately obtain

\begin{corollary}\label{cor:busparallel_foliation} Given a Ptolemy line
$\ell\subset X_\omega$.
Through any point 
$x\in X_\omega$
there is a unique Ptolemy line
$l_x$
Busemann parallel to
$\ell$.
\qed
\end{corollary}

\section{Symmetries w.r.t. horospheres}

In this section we construct a symmetry with respect to a horosphere. 

Fix 
$\omega\in X$,
a Ptolemy line
$\ell\subset X_\omega$,
and let
$c:\R\to X_\omega$
be a unit speed parameterization of
$\ell$.
For
$t>0$,
the metric sphere
$S_t=\set{x\in X_\omega}{$|xc(t)|=t$}$
passes through
$z=c(0)$
and lies between
$\omega$
and
$c(t)$.
By (sI), there is an s-inversion
$\phi_t=\phi_{\omega,c(t),S_t}:X\to X$.
By the compactness of
$X$,
s-inversions
$\phi_t$
subconverge as
$t\to\infty$
to a map
$\phi_\infty:X\to X$. 
Note that $\phi_\infty(\omega)=\omega$
because
$\phi_t(c(t))=\omega$
and
$c(t)\to\omega$
as
$t\to\infty$.

\begin{lemma}
Let $x\in H_z$,
where
$H_z\subset X_\omega$
is the horosphere through
$z\in\ell$
of the Busemann function 
$b^+(y)=\lim\limits_{t\to\infty}(|yc(t)|-t)$, $y\in X_\omega$. 
Then $\phi_\infty(x)=x$. 
\end{lemma}

\begin{proof}
 Since
$|zc(t)|=t$
for 
$t\ge 0$,
we have
$b^+(z)=0$.
Let 
$\ell_x$ 
be a line 
through
$x$
Busemann parallel to 
$\ell$  
and let $c':\R\to X_\omega$
be its unit speed parameterization with $c'(0)=x$
such that 
$b^+$
is the Busemann function associated with the ray
$c'([0,\infty))$.
Fix $\epsilon>0$ and let $x'=c'(\epsilon)$.  
Note that the function $|x'c(t)|-t$ is decreasing and tends to $b^+(x')=-\epsilon$. 
On the other hand $|x'c(0)|>0$. 
It means that there exists $t>0$ such that $|x'c(t)|-t=0$. 
Let $x_t=\phi_t(x)$. 
Since
$\phi_t(x')=x'$,
we have by Lemma~\ref{lem:inv_2}  
$$
|x_tx'|=t^2\frac{|xx'|}{|c(t)x|\cdot |c(t)x'|}=\frac{t\epsilon}{|c(t)x|}.
$$
Note that the function $|xc(t)|-t$ is decreasing and tends to $b^+(x)=0$. 
It means that $|xc(t)|\geq t$ and $|x_tx'|\leq \epsilon$. 
It follows that $|xx_t|\leq |xx'|+|x'x_t|\leq 2\epsilon$. 
Choosing $\epsilon\to 0$ we see that $\phi_t(x)\to x$ and then $\phi_\infty(x)=x$. 
\end{proof}

Now we show that
$\phi_\infty$
is an isometry of
$X_\omega$
which in addition reflects the Ptolemy line 
$\ell$
in
$z$.
For each 
$x$, $y\in X_\omega$
and every sufficiently large
$t>0$,
we have by Lemma~\ref{lem:inv_2}
$$|\phi_t(x)\phi_t(y)|=\frac{t^2|xy|}{|xc(t)||yc(t)|},$$
and
$|xc(t)|=t+b^+(x)+o(1)$, $|yc(t)|=t+b^+(y)+o(1)$.
Thus
$|\phi_\infty(x)\phi_\infty(y)|=|xy|$
for all
$x$, $y\in X_\omega$,
i.e.,
$\phi_\infty$
is an isometry. It preserves the Ptolemy line 
$\ell$
because every 
$\phi_t$
preserves the Ptolemy circle
$\sigma=l\cup\omega$,
and it reflects
$\ell$
in 
$z$
because
$\phi_\infty(z)=z$
and every
$\phi_t$
is an s-inversion of
$\sigma$.

\section{Proof of Theorem \ref{main:thm}}

\subsection{Some metric relations}

Recall that a Ptolemy space $X$ is said to be {\em Busemann flat} if for every Ptolemy circle
$\sigma\subset X$ and every point $\omega\in\sigma$, we have
$b^+ + b^- \equiv const$
for opposite Busemann functions $b^\pm \colon X_\omega\to \R$ associated with Ptolemy line
$\sigma_\omega$, see \cite{BS1}~sect.3.2.

\begin{lemma}\label{lem:flat}
 $X$ is Busemann flat. 
\end{lemma}

\begin{proof}
 Let $\ell\subset X_\omega$ be a Ptolemy line,  
and let
$c:\R\to X_\omega$
be a unit speed parameterization of
$\ell$. 
Consider the horosphere $H_o$ through $o=c(0)$ of the Busemann function 
$b^{+}(x)=\lim\limits_{t\to\infty}(|xc(t)|-t)$, $x\in X_\omega$. 
Let $b^{-}(x)=\lim\limits_{t\to\infty}(|xc(-t)|-t)$, $x\in X_\omega$, and let 
$\phi$ be the symmetry w.r.t. $H_o$. 
Note that if $x'=\phi(x)$, where $x,x'\in X_\omega$, then 
$b^{+}(x)=b^{-}(x')$. 
Indeed, 
\begin{multline*}
b^{-}(x')=
\lim\limits_{t\to\infty}(|x'c(-t)|-t)=\lim\limits_{t\to\infty}(|\phi(x)\phi(c(t))|-t)\\
=\lim\limits_{t\to\infty}(|xc(t)|-t)
=b^{+}(x).
 \end{multline*}
It follows that 
$b^{+}(z)=b^{-}(z)$
for every
$z\in H_o$. 
It means that $H_o$ is also a horosphere of the Busemann function 
$b^{-}$ and then   $b^+ + b^- \equiv const$. 
\end{proof}
\begin{corollary}
For each horosphere $H$ 
of the Busemann function 
$b^{+}$ the set $\phi(H)$ is also a horosphere 
of the Busemann function 
$b^{+}$, where $\phi$ is the symmetry w.r.t. $H_o$.  \qed 
\end{corollary}

\begin{lemma}
Let 
$\ell$, $\ell'\subset X_\omega$
be Busemann parallel lines,
$\phi:X_\omega\to X_\omega$
the symmetry which reflects
$\ell$
at
$o\in\ell$.
Then
$\phi$
reflects
$\ell'$
at 
$o'=H_o\cap\ell'$,
where
$H_o$
the horosphere of
$\ell$
through
$o$.
\end{lemma}
\begin{proof}
$H_o$
is the fixed point set of
$\phi$
and
$\phi(\ell')$
is Busemann parallel to
$\phi(\ell)=\ell$.
Thus by Lemma~\ref{lem:unique_line},
$\phi(\ell')=\ell'$.
\end{proof}

%
%
%

\begin{lemma}\label{lem:3eq}
Let  
$\ell$, $\ell'$
be Busemann parallel lines in $X_\omega$, 
and let $x, y\in \ell$, $x', y'\in \ell'$ such that 
$b(x)=b(x')$, $b(y)=b(y')$, where $b$ 
is a common Busemann function of $\ell$ and $\ell'$. 
Then $|xy|=|x'y'|$, $|xx'|=|yy'|$, $|xy'|=|yx'|$ and $|x'y|\geq |xx'|$. 
\end{lemma}
\begin{proof}
First equality is obvious, because 
$$
|xy|=|b(x)-b(y)|=|b(x')-b(y')|=|x'y'|.
$$
To prove the other two equalities 
consider the midpoint
$z\in \ell$ 
between
$x$, $y$, 
that is,
$|xz|=|zy|$. 
Let $H_x, H_y, H_z$ be horospheres of $b$ through $x, y, z$ respectively, 
and let $\phi$ be the symmetry w.r.t. $H_z$ such that $\phi(\ell)=\ell$. 
Note that $\phi(x)=y$ and $\phi(\ell')=\ell'$.  
It follows that $\phi(H_x)=H_y$.  
Moreover $\phi(x')=y'$ and $\phi(y')=x'$.  
Then we have $|xx'|=|yy'|$ and $|xy'|=|yx'|$.

Applying the Ptolemy inequality  
$|xy|\cdot|x'y'|+|xx'|\cdot|yy'|\geq |xy'|\cdot|yx'|$  
to the quadruple $(x,x',y,y')$,
we have 
\begin{equation}
|xy|^2+|xx'|^2\geq |yx'|^2. \tag{$\Diamond$}\label{diamond}\\
\end{equation}

On the other hand if $y''$ is symmetric to $y$ w.r.t. $H_x$ 
then $|xy''|=|xy|$ and $|x'y''|=|x'y|$. 
Applying the Ptolemy inequality to the quadruple $(x,x',y,y'')$, 
we have $|x'y|\cdot|xy''|+|x'y''|\cdot|xy|\geq |xx'|\cdot |yy''|$. 
It follows that $2|xy|\cdot|x'y|\geq2|xy|\cdot|xx'|$. 
Thus $|x'y|\geq |xx'|$. 
\end{proof}

Fix $a>0$ and let $\ell\in X_\omega$ be a Ptolemy line. 
Consider  $x,y\in\ell$ such that $|xy|=a/2$. 
Let $H_x$ and $H_y$ be horospheres through $x$ and $y$, and  
let $\phi_x$ and $\phi_y$ be the symmetries w.r.t. $H_x$ and $H_y$. 
Consider an isometry $\phi_y\circ \phi_x$ and note that it moves 
every point along a line Busemann parallel to $\ell$ at the distance $a$. 
We call such an isometry {\em $a$-shift along $\ell$} 
and denote it by $\eta_{a,\ell}$. 
Let $\ell'$ be a Ptolemy line (which is not necessarily Busemann parallel to $\ell$). 
It follows from Lemma~\ref{lem:busparallel_sublinear}
that $\ell'$ and $\eta_{a,\ell}(\ell')$ are Busemann parallel. 
It means that if $H_z$ is the horosphere w.r.t. $\ell'$ through $z$ 
then $\eta_{a,\ell}(H_z)$ is the horosphere  w.r.t. $\ell'$ through $\eta_{a,\ell}(z)$.

\subsection{Existence of non parallel lines}

Assume that $X$ is not M\"obius equivalent to $\wh\R$. 
\begin{lemma}\label{lem:nonparallellines}
For each $\omega,\omega'\in X$  
there exist distinct Ptolemy lines $\ell, \ell'\in X_\omega$ 
such that $\ell\cap \ell'=\{\omega'\}$.  
\end{lemma}
\begin{proof}
First of all, we find two Ptolemy circles with exactly two common points. 
Let $\sigma\subset X$ be a Ptolemy circle and let $\omega\in \sigma$. 
Since 
$X$ is not M\"obius equivalent to $\wh\R$ 
there is 
$x'\in X\setminus \sigma$.
Let  
$c:\R\to X_\omega$
be a unit speed parameterization of the Ptolemy line 
$\ell=\sigma\setminus\omega$ such that 
the horosphere $H$ of
$\ell$
through $c(0)$ contains $x'$. 
Let $z=c(1)$, $z'=c(-1)$ and $|x'z|=|x'z'|=r$. 
Consider an s-inversion $\phi$ w.r.t. $x'$, $\omega$ and 
the metric sphere $S_r=\set{x\in X_\omega}{$|x'x|=r$}$. 
It follows from Lemma~\ref{lem:inv_2} that the image $\phi(\ell)$ is a Ptolemy circle  
which intersect $\ell$ in two points $z$ and $z'$.  

Next let 
$\sigma_1$, $\sigma_2$
be the Ptolemy circles described above, 
$\sigma_1\cap \sigma_2=\{z,z'\}$. 
The lines 
$\ell_{1,z'}=\sigma_1\setminus z$, $\ell_{2,z'}=\sigma_2\setminus z\subset X_z$ 
through 
$z'$
are not Busemann parallel.  
Let $\ell_{1,\omega}, \ell_{2,\omega}$ 
be the lines in
$X_z$
through 
$\omega$
which are Busemann parallel to 
$\ell_{1,z'}, \ell_{2,z'}$ respectively. 
Note that 
$\ell'_1=(\ell_{1,\omega}\setminus\{\omega\})\cup \{z\}$ 
and
$\ell'_2=(\ell_{2,\omega}\setminus\{\omega\})\cup \{z\}$ 
are Ptolemy lines in $X_\omega$. 
Finally, the Ptolemy lines  
$\ell_1, \ell_2$ 
through 
$\omega'$ 
Busemann parallel to 
$\ell'_1, \ell'_2$ respectively
are distinct.
\end{proof}

\subsection{Homotheties preserve a foliation by horospheres}

Let 
$c:\R\to X_\omega$
be a unit speed parameterization of a Ptolemy line
$\ell\subset X_\omega$,   
$o=c(0)$, $z\in\ell$ 
and 
$H_z$ the horosphere w.r.t. $\ell$ through $z$. 

\begin{lemma}
Let $h$ be a homothety with the center $o$.  Then $h(H_z)$ 
is the horosphere w.r.t. $\ell$ through $h(z)$.
\end{lemma}

\begin{proof}
Let $x\in H_z$ and $\lambda$ be the coefficient of $h$. 
Then $\lim\limits_{t\to\infty}(|xc(t)|-|zc(t)|)=0$. 
Multiplying by $\lambda$, we have $\lim\limits_{t\to\infty}\lambda(|xc(t)|-|zc(t)|)=0$. 
It follows that 
$$
\lim\limits_{t\to\infty}(|h(x)h(c(t))|-|h(z)h(c(t))|)=\lim\limits_{t\to\infty}(|h(x)c(\lambda t)|-|h(z)c(\lambda t)|)=0
$$
and thus $h(H_z)\subset H_{h(z)}$. 
On the other hand, for each homothety $h$ we can consider a homothety $h'$ with the same center 
such that $h'\circ h=\id$. 
It means that $h(H_z)=H_{h(z)}$. 
\end{proof}

\subsection{Projection on horospheres}\label{sub:proj}

Here we assume that $X$ is not M\"obius equivalent to 
$\wh\R$, $o,\omega\in X$
and $\ell\subset X_\omega$ is a Ptolemy line through $o$. 

Let 
$H_o\subset X_\omega$
be the horosphere w.r.t. $\ell$ through
$o$. 
We define the projection $\pi_o\colon X_\omega\to H_o$ as follows: 
if $x\in X_\omega$ and $\ell_x$ is the Ptolemy line through
$x$
Busemann parallel to $\ell$  
then $\pi_o(x):=H_o\cap \ell_x$. 

\begin{proposition}\label{pro:proj}
Let $\ell'\neq\ell\subset X_\omega$ be a Ptolemy lines through $o$. 
Then $\pi_o(\ell')$ is a Ptolemy line. 
\end{proposition}
\begin{proof}
We prove that there exists $\alpha>0$ such that 
$|\pi_o(c'(t))\pi_o(c'(t'))|=\alpha|t-t'|$ for all $t,t'\in \R$, 
where 
$c':\R\to X_\omega$ 
is a unit speed parameterizations of $\ell'$
with
$c'(0)=o$. 
Let $z=c'(1)$, $z'=\pi_o(z)$ and $\alpha:=|oz'|/|oz|$. 

\begin{lemma}\label{lem:homogen}
Let $x_i=c'(t_i)$, $i=1,2,3$, where $t_1<t_2<t_3$. 
Then 
$$
\frac{|\pi_o(x_1)\pi_o(x_2)|}{|x_1x_2|}=\frac{|\pi_o(x_2)\pi_o(x_3)|}{|x_2x_3|}=\frac{|\pi_o(x_1)\pi_o(x_3)|}{|x_1x_3|}.
$$ 
\end{lemma}

\begin{proof}
Let $x_i\in \ell_i$, where $\ell$ and $\ell_i$ are Busemann parallel, 
and let $x_i\in H_i$, where $H_i$ is the horosphere of $\ell_i$, $i=1,2,3$.
 
Note that the homothety 
$h_1\colon X_\omega\to X_\omega$ with the center $x_1$ 
and the coefficient $|x_1x_3|/|x_1x_2|$ moves $x_2$ to $x_3$,
and 
$h_1(H_1)=H_1$.
It follows that $h_1(\ell_2)=\ell_3$. 
So if $y_2=H_1\cap \ell_2$ and $y_3=H_1\cap \ell_3$ then $h_1(y_2)=y_3$. 
Thus $|x_1y_3|/|x_1y_2|=|x_1x_3|/|x_1x_2|$. 
On the other hand $|x_1y_3|=|\pi_o(x_1)\pi_o(x_3)|$ and $|x_1y_2|=|\pi_o(x_1)\pi_o(x_2)|$.
It follows that 
$$
\frac{|\pi_o(x_1)\pi_o(x_2)|}{|x_1x_2|}=\frac{|\pi_o(x_1)\pi_o(x_3)|}{|x_1x_3|}. 
$$
In the same way considering the homothety $h_3$ with the center $x_3$ and the coefficient $|x_1x_3|/|x_2x_3|$ 
we obtain that 
$$
\frac{|\pi_o(x_2)\pi_o(x_3)|}{|x_2x_3|}=\frac{|\pi_o(x_1)\pi_o(x_3)|}{|x_1x_3|}. 
$$
\end{proof}

Now it follows from Lemma~\ref{lem:homogen} that 
$|\pi_o(c'(t))\pi_o(c'(t'))|=\alpha|t-t'|$ for all $t,t'\in \R$. 
\end{proof}

\subsection{Horospheres invariance}

Let $H_o\subset X_\omega$ be the horosphere through $o$ 
w.r.t. some Ptolemy line $\ell\subset X_\omega$. 

\begin{proposition}\label{pro:horospere_EI}
The subspace $X^1=H_o\cup\{\omega\}$ is a compact Ptolemy space with properties (E) and (sI). 
\qed
\end{proposition}

\begin{proof}
Let  $\phi$  be an s-inversion w.r.t. $o, \omega\in X$. 
Note that $\phi(X^1)=X^1$. 
Indeed, 
let $z\in H_o\cup\{\omega\}$, 
and let $c:\R\to X_\omega$
be a unit speed parameterization of $\ell$ such that $c(0)=o$. 
Consider the Busemann function 
$b\colon X_\omega\to \R$ 
of
$\ell$
such that $b\circ c(t)=-t$. 
Then $b(z)=0$.  
On the other hand, if $z'=\phi(z)$ then
$$
|z'c(t)|=\frac{|zc(1/t)|}{\frac1t\cdot|xz|}=\frac{t|zc(1/t)|}{|xz|}.
$$ 
Then 
$$
b(z')=
\lim\limits_{t\to \infty}(|z'c(t)|-t)=
\lim\limits_{t\to \infty}(t|zc(1/t)|/|xz|-t).
$$
Note that by \eqref{diamond}, we have
$$
|zx|^2\leq |zc(1/t)|^2\leq |zx|^2+1/t^2.
$$
Then 
$$
0\leq t|zc(1/t)|/|xz|-t\leq \sqrt{t^2+1/|zx|^2}-t. 
$$
Thus 
$b(z')=\lim\limits_{t\to \infty}(t|zc(1/t)|/|xz|-t)=0$. 

It follows that for any $x,y\in X^1$ and any s-inversion $\phi_{x,y}$ w.r.t. $x,y$ 
$\phi_{x,y}(X^1)=X^1$. 

Let $x,y\in X^1$ and $S'\subset X^1$ be a metric sphere 
between
$x$ and $y$ in $X^1$. 
Note that $S'=S\cap X^1$, where $S\subset X$ is a metric sphere 
between
$x$ and $y$ in $X$. 
We define an s-inversion $\phi'_{x,y,S'}\colon X^1\to X^1$
w.r.t. 
$x, y\in X^1$
and a metric sphere 
$S'\subset X^1$
between $x$ and $y$ 
as a restriction of 
an s-inversion $\phi_{x,y,S}\colon X\to X$
w.r.t. 
$x, y\in X$
and a metric sphere 
$S\subset X$ to $X^1$. 
It follows that $X^1$ has the property
(sI). 

On the other hand by Lemma~\ref{lem:nonparallellines} 
there exists a Ptolemy line $\ell'\neq \ell$ through $o$.   
By Proposition~\ref{pro:proj} $\pi_o(\ell')$ is a Ptolemy line in $H_o$
and  then $X^1$ has the property
(E). 
\end{proof}

\subsection{Coordinates in $X_\omega$}

From now one we fix $o,\omega\in X$ and consider a metric space $X_\omega$. 
Consider a Ptolemy line $\ell_0$ through $o$ with a unit speed parameterization
$c_0\colon \R\to X_\omega$, $c_0(0)=o$. 
Let $H_o$ be the horosphere w.r.t. $\ell_0$
through
$o$, $b_0:X_\omega\to\R$
the Busemann function of
$\ell_0$
with
$b_0(o)=0$. 
For each $z\in H_o$ 
denote by 
$\ell_z$
the line Busemann parallel to $\ell_0$ through $z$ 
and consider the unit speed parameterization
$c_z:\R\to X_\omega$ of $\ell_z$ such that $b_0\circ c_0(t)=-t=b_0\circ c_z(t)$. 
From Lemma~\ref{lem:unique_line}, Corollary~\ref{cor:busparallel_foliation} and Lemma~\ref{lem:flat}
we have that the map 
$i_1\colon \ell_0\times H_o\to X_\omega$ such that $i_1(t,z)=c_z(t)$ is a bijection. 

Take 
$x_0\in \ell_0$ 
with
$|ox_0|=1$. 
Recall that $|zx_0|\geq |ox_0|=1$ for each $z\in H_o$. 
By Proposition~\ref{pro:horospere_EI}, we have that  
$X^1=H_o\cup\{\omega\}$  is a compact Ptolemy space with properties (E) and (sI).
 
Arguing by induction we obtain
a sequence  
$$
\ldots\subset X^k\subset \ldots\subset X^1\subset X^0=X
$$
of  compact Ptolemy spaces with properties (E) and (sI) 
and a sequence of points $x_i\in X^i\setminus X^{i+1}$, 
where $|x_io|=1$. 
Moreover 
$|x_ix_k|\geq 1$
for 
$i\neq k$. 
Since the ball 
$B_1(o)=\set{x\in X}{$|xo|\leq 1$}$ 
is compact, the sequence
$\{x_i\}$
is finite and thus there exists 
$N\in\N$ 
such that 
$X^N$ is M\"obius equivalent to $\wh\R$. 
Then 
$$
\wh\R=X^N\subset \ldots\subset X^1\subset X^0=X. 
$$
It follows that there is a bijection 
$$
i\colon \ell_0\times\ell_1\times\ldots\times\ell_N\to X_\omega. 
$$
This bijection induces on $X_\omega$ a structure of the vector space $\R^{N+1}$. 
It means that we can sum up different points and multiply them by real numbers. 
Note that $o$ plays the role of a neutral element.  

Let 
$b_i:X_\omega\to\R$, $i=1,\ldots,N$,
be a  Busemann function of
$\ell_i$
with
$b_i(o)=0$. 
Then
$b_i$ 
is the i-th coordinate function. Moreover, if 
$H_i(x)$ is the horosphere w.r.t. $b_i$ through $x$ 
then $x=\bigcap\limits^N_{i=0} H_i(x)$. 
Denote by $x(i)$ the vector with coordinates $(0,\ldots,b_i(x),\ldots,0)$, where $b_i(x)$ appears at the i-th place.
Note that $x=x_0+\ldots+x_N$.  

Let 
$T^i_x\colon X_\omega\to X_\omega$ 
be the 
$b_i(x)$-shift along $\ell_i$, 
and let $T_x\colon X_\omega\to X_\omega$ be defined by $T_x(y)=x+y$, for each $y\in X_\omega$. 
Note that $T_{x(i)}=T^i_x$ and then $T_x=T^N_x\circ\ldots\circ T^0_x$. 
It follows that $T_x$ is an isometry. 

If $h_k$ is the homothety with the center $o$ and the coefficient $k$ 
then $h_k(x)=kx$, where $k>0$. 
Indeed, note that $h_k(x(i))=kx(i)$. 
Moreover, 
$$
h_k(H_i(x))=h_k(H_i(x(i))=H_i(kx(i))=H_i(kx)
$$ 
and 
$$
h_k(x)=h_k(\bigcap\limits^N_{i=0} H_i(x))=\bigcap\limits^N_{i=0} H_i(kx)=kx.
$$
It follows that $|o(kx)|=k|ox|$, where $k>0$. 

Let $\nu\colon X_\omega\to \R_+$  be defined by $\nu(x)=|ox|$. 
We prove that $\nu$ is a norm on $X_\omega$. 
Indeed, if $\nu(x)=0$  then $|ox|=0$ and $x=o$. 
Moreover, 
\begin{multline*}
\nu(x+y)=|o(x+y)|\leq |ox|+|x(x+y)|=|ox|+|T_x(o)T_x(y)|\\
=|ox|+|oy|=\nu(x)+\nu(y). 
\end{multline*}
Finally, note that $\nu(-x)=|o(-x)|=|T_x(o)T_x(-x)|=|xo|=\nu(x)$. 
So if $k\geq 0$ then $\nu(kx)=|o(kx)|=k|ox|=k\nu(x)$. 
If $k<0$ then $\nu(kx)=|o(kx)|=|o(|k|(-x))|=|k||o(-x)|=|k|\nu(-x)=|k|\nu(x)$.

Also we  note that $\nu(\cdot)$ induces the metric $X_\omega$. 
Indeed, $|xy|=|T_x(o)T_x(y-x)|=\nu(y-x)$. 
Applying the Schoenberg theorem, see Theorem~\ref{thm:schoenberg}, 
we obtain that $X$ is M\"obius equivalent to 
$\wh\R^N$. 
\qed

\end{document}